\documentclass{amsart}

\usepackage{amssymb}
\usepackage{dsfont}
\usepackage{picture}
\usepackage{color}
\usepackage[pdftex]{graphicx}

\newtheorem{theorem}{Theorem}[section]
\newtheorem{lemma}[theorem]{Lemma}
\newtheorem{proposition}[theorem]{Proposition}

\newtheorem{corollary}[theorem]{Corollary}

\newtheorem{example}[theorem]{Example}

\newcommand{\N}{\mathbb N}

\newcommand{\R}{\mathbb R}

\newcommand{\on}{\operatorname}
\author{Szymon G\l \c ab}
\address{Institute of Mathematics, \L \'od\'z University of Technology,
W\'olcza\'nska 215, 93-005 \L \'od\'z, Poland}
\email {szymon.glab@p.lodz.pl}

\author{Jacek Marchwicki}
\address{Institute of Mathematics, \L \'od\'z University of Technology,
W\'olcza\'nska 215, 93-005 \L \'od\'z, Poland}
\email {marchewajaclaw@gmail.com}

\title[On the set of limit points]{On the set of limit points of conditionally convergent series}

\thanks{The first author has been supported by the National Science Centre Poland Grant no. DEC-2012/07/D/ST1/02087.}

\begin{document}

\begin{abstract}
Let $\sum_{n=1}^\infty x_n$ be a conditionally convergent series in a Banach space and let $\tau$ be a permutation of natural numbers. We study the set $\on{LIM}(\sum_{n=1}^\infty x_{\tau(n)})$ of all limit points of a sequence $(\sum_{n=1}^p x_{\tau(n)})_{p=1}^\infty$ of partial sums of a rearranged series $\sum_{n=1}^\infty x_{\tau(n)}$. We give full characterization of limit sets in finite dimensional spaces. Namely, a limit set in $\R^m$ is either compact and connected or it is closed and all its connected components are unbounded. On the other hand each set of one of these types is a limit set of some rearranged conditionally convergent series. Moreover, this characterization does not hold in infinite dimensional spaces. 

We show that if $\sum_{n=1}^\infty x_n$ has the Rearrangement Property and $A$ is a closed subset of the closure of the $\sum_{n=1}^\infty x_n$ sum range and it is $\varepsilon$-chainable for every $\varepsilon>0$, then there is a permutation $\tau$ such that $A=\on{LIM}(\sum_{n=1}^\infty x_{\tau(n)})$. As a byproduct of this observation we obtain that series having the Rearrangement Property have closed sum ranges.
\end{abstract}

\maketitle

\section{Introduction}

Let $\sum_{n=1}^\infty x_n$ be a conditionally convergent series on the real line $\R$. For any $a<b$ one can find a permutation $\sigma\in S_\infty$ of natural numbers such that the sequence $\big(\sum_{n=1}^kx_{\sigma(n)}\big)_{k=1}^\infty$ of partial sums of the rearrangement $\sum_{n=1}^\infty x_{\sigma(n)}$ oscillates between $a$ and $b$. Consequently, $a$ and $b$ are limit points of a sequence of rearranged partial sums  $\big(\sum_{n=1}^kx_{\sigma(n)}\big)_{k=1}^\infty$. Since $\vert x_{\sigma(n)}\vert$ tends to zero, the whole interval $[a,b]$ consists of limit points of $\big(\sum_{n=1}^kx_{\sigma(n)}\big)_{k=1}^\infty$. This simple observation shows that the set of all limit points of a sequence of  rearranged partial sums $\big(\sum_{n=1}^kx_{\sigma(n)}\big)_{k=1}^\infty$ is closed and connected, and for any close connected subset $I$ of real line and any conditionally convergent series $\sum_{n=1}^\infty x_n$ one can find a rearrangement $\sum_{n=1}^\infty x_{\sigma(n)}$ such that the set of all limit points of its partial sums equals $I$. If the rearranged series $\sum_{n=1}^\infty x_{\sigma(n)}$ converges to $\infty$ or to $-\infty$, then the set of all its limit points is empty. 

The situation becomes more complicated if limit sets of rearrangements of conditionally convergent series are considered in multidimensional Euclidean spaces. One could expect that such limit sets would be connected or even arcwise connected. It turns out that this is not the case. However, some result concerning connectedness  can be proved for multidimensional spaces, see Theorem \ref{CharacterizationOfLimitSets}. 

Now, let $\sum_{n=1}^\infty x_n$ be a conditionally convergent series in the Euclidean space $\R^n$. By Steinitz Theorem the sum range $\on{SR}(\sum_{n=1}^\infty x_n)=\{\sum_{n=1}^\infty x_{\sigma(n)}:\sigma\in S_\infty\}$ of $\sum_{n=1}^\infty x_n$, where $S_\infty$ is a symmetric group of all permutation of natural numbers, is an affine subspace of $\R^n$. Denote by $\on{LIM}(\sum_{n=1}^\infty x_{\sigma(n)})$ the set of all limit points of a sequence of rearranged partial sums $\big(\sum_{n=1}^kx_{\sigma(n)}\big)_{k=1}^\infty$. Such limit sets were studied by Victor Klee in \cite{Klee}, where the author claimed that if $A$ is a limit set $\on{LIM}(\sum_{n=1}^\infty x_{\sigma(n)})$, then for every $\varepsilon>0$ an $\varepsilon$-shell $A(\varepsilon)=\{x:\Vert x-y\Vert<\varepsilon$ for some $y\in A\}$ of $A$ is connected. Our Example \ref{Ex1} shows that this claim is not true. Note that connectedness of $A(\varepsilon)$ means that any two points $a,b\in A$ can be joined by a path $x_0,x_1,\dots,x_k\in A$ such that $x_0=a$, $x_k=b$ and $\Vert x_i-x_{i-1}\Vert<\varepsilon$, and if $A$ has this property, then we say that $A$ is $\varepsilon$-chainable. Klee also proved that if $A\subseteq\on{SR}(\sum_{n=1}^\infty x_n)$ is closed and $\varepsilon$-chainable for every $\varepsilon>0$, then there is $\sigma\in S_\infty$ such that $A=\on{LIM}(\sum_{n=1}^\infty x_{\sigma(n)})$. 

In this article we complete the Klee's result by giving the full characterization of limit sets $\on{LIM}(\sum_{n=1}^\infty x_{\sigma(n)})$ in Euclidean spaces. Namely we prove the following dichotomy (Theorem \ref{CharacterizationOfLimitSets}): the limit set is either compact and connected or any its component is unbounded; moreover, the closure of the limit set in the one-point compactification of $\R^m$ is connected. The proof uses the fact that underlying space has a finite dimension. Moreover, this dichotomy does not hold for all Banach spaces. More precisely, we construct an example of a conditionally convergent series in $c_0$ such that the limit set of some of its rearrangement consists of two points. 

Theorem \ref{CharacterizationOfLimitSets} cannot be reversed in the sense that there is an unbounded, closed set in the one-dimensional  Euclidean space $\R$ whose every component is unbounded but it cannot be a limit set. Namely, consider the union $X:=(-\infty,-1]\cup[1,\infty)$ of two unbounded connected sets. As we have mentioned in the beginning, any limit set on the real line must be connected, and therefore $X$ is not a limit set. However, Theorem \ref{CharacterizationOfLimitSets} can be reversed in higher dimensions. This means that any compact connected set (or even any closed $\varepsilon$-chainable set for every $\varepsilon>0$) in $\R^m$, $m\geq 1$, and any closed set in $\R^m$, $m\geq 2$, whose every component is unbounded are limit sets of some rearrangement of a conditionally convergent series. 

In the last Section we show that if $\sum_{n=1}^\infty x_n$ has the Rearrangement Property and $A\subseteq\overline{\on{SR}(\sum_{n=1}^\infty x_n)}$ is closed and $\varepsilon$-chainable for every $\varepsilon>0$, then there is $\tau\in S_\infty$ such that $A=\on{LIM}(\sum_{n=1}^\infty x_{\tau(n)})$. As a byproduct of this observation we obtain that series having the Rearrangement Property have closed sum ranges.

\section{Counterexample for Klee's claim}

As we have mentioned in the Introduction, Victor Klee in \cite{Klee} claimed that if $A=\on{LIM}(\sum_{n=1}^\infty x_{\sigma(n)})$, then its $\varepsilon$-shell $A(\varepsilon)=\{x:\Vert x-y\Vert<\varepsilon$ for some $y\in A\}$ is connected for every $\varepsilon>0$. It is equivalent to saying that $A$ is $\varepsilon$-chainable for every $\varepsilon>0$.  The author used quite a different notation than the one used by us, but the gap in his argument can be translated into our language as follows. Klee argued that $\on{LIM}(\sum_{n=1}^\infty x_{\sigma(n)})$ cannot intersect two sets $X$ and $Y$ having disjoint $\varepsilon$-shells $X(\varepsilon)$ and $Y(\varepsilon)$; it is supposed to be ''evident''. However, the following example shows that this is simply not true. 

For natural numbers $n<m$ by $[n,m]$ we denote discrete interval $\{n,n+1,n+2,\dots,m\}$ and by $[n,\infty)$ we denote the set $\{n,n+1,\dots\}$. Let $\sum_{n=1}^\infty y_n$ be a conditionally convergent series and let $\sum_{n=1}^\infty x_n$ be its rearrangement, that is there is $\sigma\in S_\infty$ with $x_n=y_{\sigma(n)}$. A partial sums sequence $(s_n)$, $s_n=\sum_{k=1}^nx_k$, will be called a \emph{walk}. Note that $a\in\on{LIM}(\sum_{n=1}^\infty x_n)$ if for every $\varepsilon>0$ the walk $(s_n)$ hits the ball $B(a,\varepsilon)$. If $(s_n)_{n=1}^\infty$ is a sequence in $\R^m$, then we call it a walk, if some rearrangement of a series $\sum_{n=1}^\infty(s_{n+1}-s_n)$ is convergent. 

A sequence $(s_n)$ of elements of set $X$ is called an $X$-walk if\\
(i) the set $\{s_n:n\in\N\}$ is dense in $X$;\\
(ii) there are positive integers $n_1,n_2,\dots$ such that $s_{n_1+i}=s_{n_1-i}$ for $i\in[1,n_1-1]$ and 
$$
s_{\sum_{j=1}^{k}2n_j-1+n_{k+1}-i}=s_{\sum_{j=1}^{k}2n_j-1+n_{k+1}+i}
$$ 
for $k>0$ and $i\in[1,n_{k+1}]$;\\
(iii) $\Vert s_{n+1}-s_n\Vert\to 0$.

\begin{proposition}\label{X-walkProp}
Suppose that $(s_n)$ is an $X$-walk. Then there is a conditionally convergent series $\sum_{n=1}^\infty x_n$ and a permutation $\sigma\in S_\infty$ such that $s_n=\sum_{k=1}^nx_{\sigma(k)}$. Moreover, $\on{LIM}(\sum_{k=1}^\infty x_{\sigma(k)})=X$. 
\end{proposition}

\begin{proof}
Note that $s_{\sum_{j=1}^k2n_j-1}=s_1$ for every $k$. That means that the $X$-walk $(s_n)$ gets from $s_1$ to $s_{n_1}$ and back, using the same points, to $s_{2n_1-1}=s_1$, then it walks to $s_{2n_1-1+n_2}$ and back to $s_{2n_1+2n_2-1}=s_1$, and so on.  We define $y_n=s_{n+1}-s_n$. Let $p=\sum_{j=1}^{k}2n_j-1+n_{k+1}-i$ and $m=\sum_{j=1}^{k}2n_j-1+n_{k+1}+i$. Then $y_{p}=-y_{m-1}$. Thus the series $\sum_{n=1}^\infty y_n$ can be rearranged into an alternating series $\sum_{n=1}^\infty x_n$, which by (iii) is convergent. Since each element of $(s_n)$ is in a closed set $X$, then $\on{LIM}(\sum_{k=1}^\infty x_{\sigma(k)})\subseteq X$. The opposite inclusion follows from (i). 
\end{proof}

\begin{example} \label{Ex1}
\emph{ At first we define elements $y_{n}$ of a conditionally convergent series $\sum_{n=1}^{\infty}y_{n}\subseteq\R^2$. The sequence $(y_{n})$ will be alternating, that is $y_{2n}=-y_{2n-1}$ for $n\geq 1$. Therefore we will define only the terms $y_{n}$ with an odd index $n$.\\
\emph{Step 1.} First two odd $y_{n}'s$ are $(\frac{1}{2},0),(\frac{1}{2},0)$.\\
\emph{Step 2.} We define next $1\cdot 4+4+1\cdot 4$ odd elements: $(0,\frac{1}{4}),\ldots,(0,\frac{1}{4})$,$(\frac{1}{4},0),\ldots,(\frac{1}{4},0)$,$(0,-\frac{1}{4}),\ldots,(0,-\frac{1}{4})$.\\
\emph{Step k+1.} In this step we define $k\cdot 2^{k+1}+ 2^{k+1}+k\cdot 2^{k+1}$ elements 
$$
\underbrace{(0,\frac{1}{2^{k+1}}),\ldots,(0,\frac{1}{2^{k+1}})}_{k\cdot 2^{k+1}},\underbrace{(\frac{1}{2^{k+1}},0),\ldots,(\frac{1}{2^{k+1}},0)}_{2^{k+1}},\underbrace{(0,-\frac{1}{2^{k+1}}),\ldots,(0,-\frac{1}{2^{k+1}})}_{k\cdot 2^{k+1}}.
$$
Since $(y_{n})$  is alternating and $\lim _{n\to\infty}\Vert y_{n}\Vert =0$, the series $\sum _{n=1}^{\infty}y_{n}$ is convergent. Now we define our walk, that is a rearrangement of $\sum _{n=1}^{\infty}y_{n}$, as follows. First two $x_{1}$ and $x_{2}$ are the elements of $(y_{n})$ defined in Step 1 with odd indexes, $x_{3}, x_{4}$ are corresponding elements of $(y_{n})$  with even indexes. Next $1\cdot 4+4+1\cdot 4$ of $x_{n}'s$ are elements of $(y_{n})$ defined in Step 2 with odd indexes (in the same order we have defined them above) and the next $1\cdot 4+4+1\cdot 4$ of $x_{n}'s$ are corresponding elements of $(y_{n})$ with even indexes taken with reversed order, and so on. On Figure 1 
we present a sequence of partial sums given for $x_n$'s defined in the first three steps of the construction. Note that $\on{LIM}(\sum _{n=1}^{\infty} x_{n})=\{0,1\}\times [0,\infty)$. Thus the set of limit points of the series $\sum _{n=1}^{\infty} x_{n}$ has no connected $\varepsilon$-shell for $\varepsilon<\frac{1}{2}$. }
\end{example}

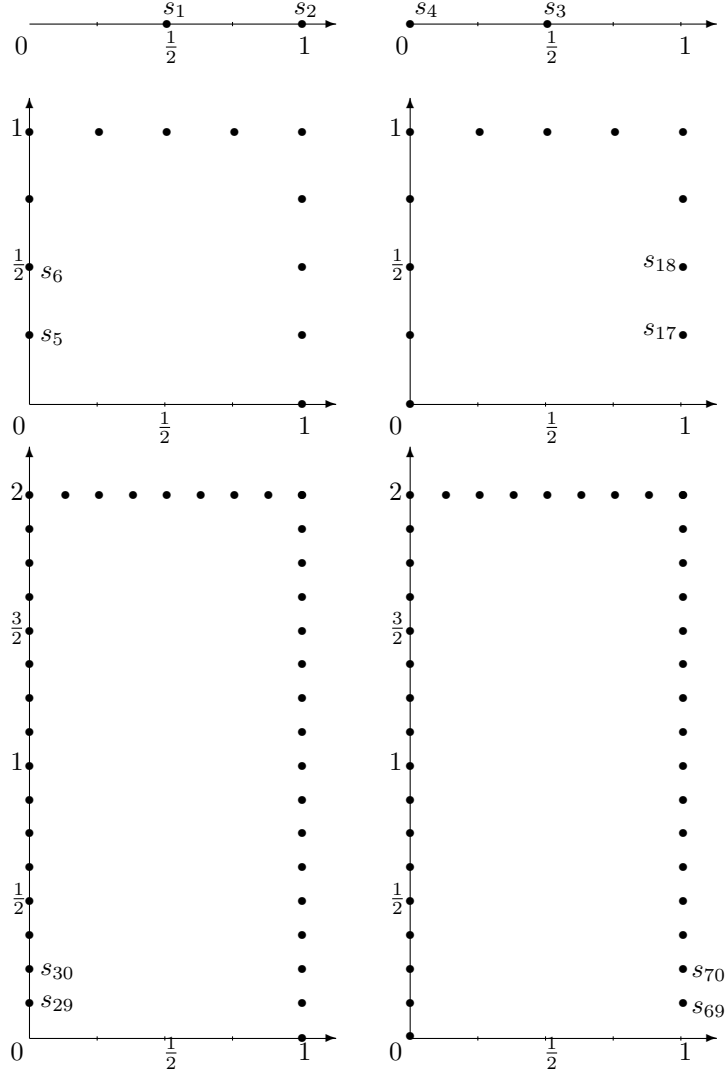
\begin{figure}[h]
\begin{center}
\setlength{\unitlength}{0.8 pt} 
\begin{picture}(400,500)(0,0) 

\put(10,310){\vector(1,0){145}}
\put(10,310){\vector(0,1){145}}

\put(9,438){\line(1,0){2}}
\put(1,437){\mbox{$1$}}
\put(1,373){\mbox{$\frac{1}{2}$}}
\put(9,342){\line(1,0){2}}
\put(9,374){\line(1,0){2}}
\put(9,406){\line(1,0){2}}
\put(2,296){\mbox{$0$}}
\put(138,309){\line(0,1){2}}
\put(137,296){\mbox{$1$}}
\put(70,296){\mbox{$\frac{1}{2}$}}
\put(42,309){\line(0,1){2}}
\put(74,309){\line(0,1){2}}
\put(106,309){\line(0,1){2}}
\put(15,370){\mbox{$s_{6}$}}
\put(15,340){\mbox{$s_{5}$}}
\put(10,343){\circle*{4}}
\put(10,375){\circle*{4}}
\put(10,407){\circle*{4}}
\put(10,439){\circle*{4}}
\put(43,439){\circle*{4}}
\put(75,439){\circle*{4}}
\put(107,439){\circle*{4}}
\put(139,439){\circle*{4}}
\put(139,407){\circle*{4}}
\put(139,375){\circle*{4}}
\put(139,343){\circle*{4}}
\put(139,310){\circle*{4}}

\put(190,310){\vector(1,0){145}}
\put(190,310){\vector(0,1){145}}

\put(189,438){\line(1,0){2}}
\put(180,437){\mbox{$1$}}
\put(180,373){\mbox{$\frac{1}{2}$}}
\put(189,342){\line(1,0){2}}
\put(189,374){\line(1,0){2}}
\put(189,406){\line(1,0){2}}
\put(180,296){\mbox{$0$}}
\put(318,309){\line(0,1){2}}
\put(317,296){\mbox{$1$}}
\put(253,296){\mbox{$\frac{1}{2}$}}
\put(222,309){\line(0,1){2}}
\put(254,309){\line(0,1){2}}
\put(286,309){\line(0,1){2}}
\put(190,343){\circle*{4}}
\put(190,375){\circle*{4}}
\put(190,407){\circle*{4}}
\put(190,439){\circle*{4}}
\put(223,439){\circle*{4}}
\put(255,439){\circle*{4}}
\put(287,439){\circle*{4}}
\put(319,439){\circle*{4}}
\put(319,407){\circle*{4}}
\put(319,375){\circle*{4}}
\put(319,343){\circle*{4}}
\put(190,310){\circle*{4}}
\put(300,343){\mbox{$s_{17}$}}
\put(300,375){\mbox{$s_{18}$}}

\put(10,490){\vector(1,0){145}}

\put(138,489){\line(0,1){2}}
\put(137,476){\mbox{$1$}}
\put(73,476){\mbox{$\frac{1}{2}$}}
\put(3,476){\mbox{$0$}}
\put(42,489){\line(0,1){2}}
\put(74,489){\line(0,1){2}}
\put(106,489){\line(0,1){2}}
\put(73,495){\mbox{$s_{1}$}}
\put(135,495){\mbox{$s_{2}$}}
\put(75,490){\circle*{4}}
\put(139,490){\circle*{4}}

\put(190,490){\vector(1,0){145}}

\put(183,476){\mbox{$0$}}
\put(318,489){\line(0,1){2}}
\put(317,476){\mbox{$1$}}
\put(253,476){\mbox{$\frac{1}{2}$}}
\put(222,489){\line(0,1){2}}
\put(254,489){\line(0,1){2}}
\put(286,489){\line(0,1){2}}
\put(255,490){\circle*{4}}
\put(190,490){\circle*{4}}
\put(253,495){\mbox{$s_{3}$}}
\put(192,495){\mbox{$s_{4}$}}

\put(10,10){\vector(1,0){145}}
\put(10,10){\vector(0,1){280}}

\put(9,138){\line(1,0){2}}
\put(1,137){\mbox{$1$}}
\put(1,73){\mbox{$\frac{1}{2}$}}
\put(0,201){\mbox{$\frac{3}{2}$}}
\put(1,265){\mbox{$2$}}
\put(9,42){\line(1,0){2}}
\put(9,74){\line(1,0){2}}
\put(9,106){\line(1,0){2}}
\put(1,0){\mbox{$0$}}
\put(138,9){\line(0,1){2}}
\put(137,0){\mbox{$1$}}
\put(73,0){\mbox{$\frac{1}{2}$}}
\put(42,9){\line(0,1){2}}
\put(74,9){\line(0,1){2}}
\put(106,9){\line(0,1){2}}
\put(15,25){\mbox{$s_{29}$}}
\put(15,40){\mbox{$s_{30}$}}
\put(10,27){\circle*{4}}
\put(10,43){\circle*{4}}
\put(10,59){\circle*{4}}
\put(10,75){\circle*{4}}
\put(10,91){\circle*{4}}
\put(10,107){\circle*{4}}
\put(10,123){\circle*{4}}
\put(10,139){\circle*{4}}
\put(10,155){\circle*{4}}
\put(10,171){\circle*{4}}
\put(10,187){\circle*{4}}
\put(10,203){\circle*{4}}
\put(10,219){\circle*{4}}
\put(10,235){\circle*{4}}
\put(10,251){\circle*{4}}
\put(10,267){\circle*{4}}
\put(139,27){\circle*{4}}
\put(139,43){\circle*{4}}
\put(139,59){\circle*{4}}
\put(139,75){\circle*{4}}
\put(139,91){\circle*{4}}
\put(139,107){\circle*{4}}
\put(139,123){\circle*{4}}
\put(139,139){\circle*{4}}
\put(139,155){\circle*{4}}
\put(139,171){\circle*{4}}
\put(139,187){\circle*{4}}
\put(139,203){\circle*{4}}
\put(139,219){\circle*{4}}
\put(139,235){\circle*{4}}
\put(139,251){\circle*{4}}
\put(139,267){\circle*{4}}
\put(139,10){\circle*{4}}
\put(27,267){\circle*{4}}
\put(43,267){\circle*{4}}
\put(59,267){\circle*{4}}
\put(75,267){\circle*{4}}
\put(91,267){\circle*{4}}
\put(107,267){\circle*{4}}
\put(123,267){\circle*{4}}
\put(139,267){\circle*{4}}

\put(190,10){\vector(1,0){145}}
\put(190,10){\vector(0,1){280}}

\put(189,138){\line(1,0){2}}
\put(180,137){\mbox{$1$}}
\put(180,73){\mbox{$\frac{1}{2}$}}
\put(180,201){\mbox{$\frac{3}{2}$}}
\put(180,265){\mbox{$2$}}
\put(189,42){\line(1,0){2}}
\put(189,74){\line(1,0){2}}
\put(189,106){\line(1,0){2}}
\put(180,0){\mbox{$0$}}
\put(318,9){\line(0,1){2}}
\put(317,0){\mbox{$1$}}
\put(253,0){\mbox{$\frac{1}{2}$}}
\put(222,9){\line(0,1){2}}
\put(254,9){\line(0,1){2}}
\put(286,9){\line(0,1){2}}
\put(323,22){\mbox{$s_{69}$}}
\put(323,39){\mbox{$s_{70}$}}
\put(190,11){\circle*{4}}
\put(190,27){\circle*{4}}
\put(190,43){\circle*{4}}
\put(190,59){\circle*{4}}
\put(190,75){\circle*{4}}
\put(190,91){\circle*{4}}
\put(190,107){\circle*{4}}
\put(190,123){\circle*{4}}
\put(190,139){\circle*{4}}
\put(190,155){\circle*{4}}
\put(190,171){\circle*{4}}
\put(190,187){\circle*{4}}
\put(190,203){\circle*{4}}
\put(190,219){\circle*{4}}
\put(190,235){\circle*{4}}
\put(190,251){\circle*{4}}
\put(190,267){\circle*{4}}
\put(319,27){\circle*{4}}
\put(319,43){\circle*{4}}
\put(319,59){\circle*{4}}
\put(319,75){\circle*{4}}
\put(319,91){\circle*{4}}
\put(319,107){\circle*{4}}
\put(319,123){\circle*{4}}
\put(319,139){\circle*{4}}
\put(319,155){\circle*{4}}
\put(319,171){\circle*{4}}
\put(319,187){\circle*{4}}
\put(319,203){\circle*{4}}
\put(319,219){\circle*{4}}
\put(319,235){\circle*{4}}
\put(319,251){\circle*{4}}
\put(319,267){\circle*{4}}
\put(207,267){\circle*{4}}
\put(223,267){\circle*{4}}
\put(239,267){\circle*{4}}
\put(255,267){\circle*{4}}
\put(271,267){\circle*{4}}
\put(287,267){\circle*{4}}
\put(303,267){\circle*{4}}
\put(319,267){\circle*{4}}
\end{picture}
\end{center}
\label{rys}
\caption{The first three steps of the construction of the walk $(\sum _{n=1}^{m} x_{n})_{m=1}^\infty$ from Example \ref{Ex1}.}
\end{figure}

\begin{example}\label{Ex2}
\emph{Now we describe a construction in which the limit points of the series are the closure of set of infinitely many pairwise disjoint half-lines $\{a_n:n\in\N\}\times[0,\infty)$ where $(a_{n})$ is a sequence of distinct real numbers. This example is similar to Example \ref{Ex1}, so we prescribe only the walk $(s_n)$. Since in each step of the construction the walk goes from one point to another and then back through the same path, the steps of the walk can be rearranged to an alternating series. Since the lengths of the walk's steps tend to zero, the obtained series is convergent. We describe the first three steps of the construction:\\  
\emph{Step 1.}  We start the walk at $(a_{1},0)$. Then we move to $(a_{2},0)$ along the line $y=0$ using steps of length not greater than $1$. Then we go back to $(a_{1},0)$ via the same path.\\
\emph{Step 2.} We go upward to $(a_{1},1)$,  then along the line $y=1$ to $(a_{2},1)$, next downward to $(a_{2},0)$ and back upward to $(a_{2},1)$, then again along $y=1$ to $(a_{3},1)$ and downward to $(a_{3},0)$ in each part using steps of length not greater than $\frac{1}{2}$. Finally we go back to $(a_{1},0)$ using the same path.\\
\emph{Step 3.} In this step first four points $(a_{1},0),\ldots,(a_{4},0)$ are involved, steps are not greater than $\frac{1}{4}$ and to move between vertical lines $x=a_{i}$ we use a horizontal line $y=2$, etc.}

\emph{ Clearly $\on{LIM}(\sum _{n=1}^{\infty}x_{n})\supseteq \{a_{n}:n\in\N\}\times [0,\infty)$. Since $\on{LIM}(\sum _{n=1}^{\infty}x_{n})$ is closed, we obtain  $\on{LIM}(\sum _{n=1}^{\infty}x_{n})\supseteq\overline{ \{a_{n}:n\in\N\}\times [0,\infty)}=\overline{ \{a_{n}:n\in\N\}}\times [0,\infty)$. To show the inverse inclusion let $(u,v)\notin \overline{ \{a_{n}:n\in\N\}}\times [0,\infty)$. If $v<0$ then $(u,v)\notin\on{LIM}(\sum _{n=1}^{\infty}x_{n})$, because our walk is in $\R^{2}$ and has a non-negative second coordinate. If $v\geq 0$ and $u\notin\overline{ \{a_{n}:n\in\N\}}$ then $\inf _{n\in\mathbb{N}} |u-a_{n}|=\delta>0$. Fix a natural number $m>v+\delta$. Then the ball $B((u,v),\delta)$ does not contain elements of our walk defined in $k$-th step of construction for any $k\geq m$. Hence $(u,v)\notin\on{LIM}(\sum _{n=1}^{\infty}x_{n})$. Finally  $\on{LIM}(\sum _{n=1}^{\infty}x_{n})=\overline{\{a_{n}:n\in\N\}\times [0,\infty)}$.
}
\end{example}

Using Example \ref{Ex2} we can show that the limit set of a rearrangement of a conditionally convergent series can have uncountably many unbounded components. Let $E=\{a_{1},a_{2},\ldots\}$ be a countable dense subset of the ternary Cantor set $C$. By Example \ref{Ex2} one can find a conditionally convergent series $\sum _{n=1}^{\infty} x_{n}$ and a rearrangement $\sigma$ such that $\on{LIM}(\sum _{n=1}^{\infty} x_{\sigma(n)})=\overline{\{a_{n}\}_{n=1}^{\infty}\times [0,\infty)}= C\times [0,\infty)$. Since the ternary Cantor set $C$ is totally disconnected, i.e. each its component is a singleton, half-lines $\{x\}\times[0,\infty)$, $x\in C$, are the components of $C\times[0,\infty)$.

\section{Characterization of limit sets $\on{LIM}(\sum_{n=1}^\infty x_{\sigma(n)})$ in the Euclidean spaces}

Let $B(0,R)=\{v\in\R^m:\Vert v\Vert\leq R\}$ and let $S(0,R)=\{v\in\R^m:\Vert v\Vert=R\}$. For a topological space $X$ by $\mathcal K(X)$ we denote the set of all non-empty compact subsets of $X$ equipped with the Vietoris topology, for details see for example \cite[p. 66]{Sri}. It is well-known that the compactness (metrizability, separability) of $X$ implies the  compactness (metrizability, separability) of the hyperspace $\mathcal K(X)$ and that the family of all nonempty compact connected subsets of $X$ forms a closed subset of $\mathcal K(X)$.

\begin{lemma}\label{LemmaComponentsInBall}
Let $X\subseteq\R^m$ be a closed set and let $R>0$. Then
$$
Z:=\bigcup\{C:C\text{ is a component of }X\cap B(0,R)\text{ such that }C\cap S(0,R)\neq\emptyset\}
$$
is a compact subset of $\R^m$.
\end{lemma}

\begin{proof}
Let $(v_n)\subseteq Z$. Find components $C_n$ of $X\cap B(0,R)$ such that $C_n\cap S(0,R)\neq\emptyset$ and $v_n\in C_n$. Pick $x_n\in C_n\cap S(0,R)$. Since $\mathcal K(X\cap B(0,R))$ is compact, we may assume that $C_n$ tends to some $C$, $v_n\to v$ and $x_n\to x$. Then $v,x\in C$ and $C$ is connected. Therefore $v$ and $x$ are in the same component of $X\cap B(0,R)$ which has a non-empty intersection with the sphere $S(0,R)$. Thus $v\in Z$, and consequently $Z$ is compact.
\end{proof}

Let $X\subseteq\R^m$ be a closed set. We define an equivalence relation $E$ on $X$ as follows
$$
xEy\iff x\text{ and }y\text{ belong to the same component of }X.
$$
By $X/E$ we denote the set of all equivalence classes of $E$ and by $q$ we denote the mapping from $X$ to $X/E$ assigning to a point $x\in X$ the equivalence class $[x]_E\in X/E$. On $X/E$ we consider the so-called \emph{quotient topology} consisting of those $U\subseteq X/E$ such that $q^{-1}(U)$ is open in $X$. The set $X/E$ equipped with this topology is called the \emph{quotient space}, and $q:X\to X/E$ is called the \emph{natural quotient mapping}. The following result important for us can be found in \cite{Engelking}.

\begin{theorem}\label{EngelkingQuotient}
For every compact space $X$, the quotient space $X/E$ is compact and zero-dimensional.  
\end{theorem}

For $X\subseteq\R^m$ and $\varepsilon>0$ put $X(\varepsilon):=\{y\in\R^m:\Vert x-y\Vert\leq\varepsilon$ for some $x\in X\}$. We will called it an $\varepsilon$-shell or an $\varepsilon$-neighborhood of $X$. 

\begin{lemma}\label{LemBoundedSeparatedComponnent}
Let $\sum_{n=1}^\infty x_n$ be a conditionally convergent series in $\R^m$ and let $\sigma\in S_\infty$. Assume that $Y$ is a nonempty bounded subset of $X:=\on{LIM}(\sum_{n=1}^\infty x_{\sigma(n)})$. If $Y(\varepsilon)$ is disjoint with $X\setminus Y$ for some $\varepsilon>0$, then $X=Y$.
\end{lemma}

\begin{proof}
Note that the closure $Z$ of $Y(\varepsilon)\setminus Y(\varepsilon/2)$ is a compact set disjoint with $X$. Suppose that $X\setminus Y\neq\emptyset$. Consider a set $A:=\{\sum_{n=1}^k x_{\sigma(n)}:k\in\N\}\cap Z$ of those partial sums of $\sum_{n=1}^\infty x_{\sigma(n)}$ which meet $Z$. Since all elements of the nonempty sets $Y$ and $X\setminus Y$ are limit points of a rearranged partial sums sequence $\{\sum_{n=1}^kx_{\sigma(n)}\}_{k=1}^\infty$, then the elements of that sequence walk from $Y$ to $X\setminus Y$ and back infinitely many times. Since the lengths of steps $\Vert x_{\sigma(n)}\Vert$ taken during this walk tend to zero, the set $A$ is infinite. By compactness of $Z$ we obtain that $A$ has a limit point, which in turn is in $Z$, but this contradicts the fact that $Z\cap X=\emptyset$. Thus $X\setminus Y=\emptyset$ and consequently $X=Y$. 
\end{proof}

By $a(\R^m)$ denote the one-point compactification of $\R^m$, that is to the underlying set $\R^m$ we add a point $\infty$. Neighborhood base at each $x\in\R^m$ consists of open ball centered in $x$ and  neighborhood base at $\infty$ consists of all sets of the form $(\R^m\setminus C)\cup\{\infty\}$ where $C$ is compact in $\R^m$. For $A\subset a(\R^m)$ by $\overline{A}^\infty$ denote the closure of $A$ in $a(R^m)$. 

\begin{lemma}\label{ClosureIn1PointCompactification}
Let $\{C_i:i\in I\}$ be a family of connected and unbounded subsets of $\R^m$ and let $C:=\bigcup_{i\in I}C_i$. Then\\
(1) $\overline{C}^\infty=\overline{C}\cup\{\infty\}$;\\
(2) $\overline{C}^\infty$ is connected. 
\end{lemma}

\begin{proof} (1) The set $\overline{C}\cup\{\infty\}$ is closed in $a(\R^m)$, since $(\R^m\cup\{\infty\})\setminus(\overline{C}\cup\{\infty\})=\R^m\setminus\overline{C}$ is open in $\R^m$. Thus $\overline{C}^\infty\subseteq\overline{C}\cup\{\infty\}$. Since $C$ is unbounded, then $\infty\in\overline{C}^\infty$, and consequently $\overline{C}\cup\{\infty\}\subseteq\overline{C}^\infty$. 

(2) Note that $\overline{C\cup\{\infty\}}^\infty=\overline{C}\cup\{\infty\}$ -- it follows from (1) and inclusions $C\subseteq C\cup\{\infty\}\subseteq\overline{C}\cup\{\infty\}$. It is enough to show that $A:=C\cup\{\infty\}$ is connected. Suppose to the contrary that there are two nonempty disjoint open sets $U$ and $V$ with $A=(A\cap U)\cup(A\cap V)$ and $\infty\in U$. Put $U':=U\setminus\{\infty\}$. Then $U'$ is open in $\R^m$. There is a compact set $D\subseteq\R^m$ such that $(X\setminus D)\cup\{\infty\}=U$. Then $X\setminus D=U'$ and $V\subseteq D$. Since $V$ is nonempty, there is $i\in I$ with $V\cap C_i\neq\emptyset$. But then $C_i=(V\cap C_i)\cup(U\cap C_i)$ and by the connectedness of $C_i$ we obtain $C_i\subseteq V\subseteq D$ which contradicts the unboundedness of $C_i$. 
\end{proof}

\begin{theorem}\label{CharacterizationOfLimitSets}
Let $\sum_{n=1}^\infty x_n$ be a conditionally convergent series in $\R^m$ and let $\sigma\in S_\infty$ be a permutation of indexes. Then the set $X=\on{LIM}(\sum_{n=1}^\infty x_{\sigma(n)})$ is either compact connected or it is a union (finite, infinite countable or uncountable) of unbounded closed connected sets; in particular $\overline{X}^\infty$ is compact and connected.  
\end{theorem}

\begin{proof}
Since $X$ is closed in $\R^m$, $\overline{X}^\infty=X$ if $X$ is bounded and $\overline{X}^\infty=X\cup\{\infty\}$ if $X$ is unbounded. On $\overline{X}^\infty$ define an equivalence relation $E$ given by the decomposition of $\overline{X}^\infty$ into components. 

Assume that $C$ is a bounded component of $X$. There is $R>0$ such that $C\subseteq B(0,R)$ and $C\cap S(0,R)=\emptyset$. Suppose that the set $Z:=\bigcup\{C':C'$ is component of $X\cap B(0,R)$ such that $C'\cap S(0,R)\neq\emptyset\}$ is nonempty. Then, by Lemma \ref{LemmaComponentsInBall}, $Z$ is compact in $\R^m$. Put $U:=(\overline{X}^\infty\cap B(0,R))\setminus Z$. Then $U$ is open in $\overline{X}^\infty$ and $U=q^{-1}(q(U))$; therefore $q(U)$ is open in $\overline{X}^\infty/E$. Since $C\in q(U)$ and $\overline{X}^\infty/E$ is zero-dimensional, there is a clopen set $V\subseteq\overline{X}^\infty/E$ with $C\in V\subseteq q(U)$. Since $Z$ and $Y:=q^{-1}(V)$ are compact, there is $\varepsilon>0$ with $Y(\varepsilon)\cap Z=\emptyset$ and $(Y(\varepsilon)\setminus Y)\cap\overline{X}^\infty=\emptyset$; consequently $Y(\varepsilon)\cap X\setminus Y=\emptyset$. By Lemma \ref{LemBoundedSeparatedComponnent} we obtain that $Z\subseteq X\setminus Y=\emptyset$ which gives a contradiction. Thus $Z=\emptyset$. Therefore there are no components of $X$ having nonempty intersection with $S(0,R)$. Thus $X\cap B(0,R)=q^{-1}\big(q(X\cap B(0,R))\big)$ and consequently $q(X\cap B(0,R))$ is open in $\overline{X}^\infty/E$. Since $\overline{X}^\infty/E$ is zero dimensional, there is a clopen $V$ with $C\in V\subseteq q(X\cap B(0,R))$. Thus $Y:=q^{-1}(V)$ is clopen and it contains $C$. There is $\varepsilon>0$ such that $Y(\varepsilon)\subseteq B(0,R)$ which means that $Y(\varepsilon)$ is disjoint with $X\setminus Y$. By Lemma \ref{LemBoundedSeparatedComponnent} we obtain that $X$ is bounded. 

We have already proved that if $C$ is a bounded component of $X$, then $X$ is bounded itself. That means that if $X$ has an unbounded component, then each its component is unbounded and, by Lemma \ref{ClosureIn1PointCompactification}, $\overline{X}^\infty$ is connected -- equivalently $q(\overline{X}^\infty)=[\infty]_E$. Thus $\overline{X}^\infty$ is connected in $a(\R^m)$ if $X$ is unbounded. To finish the proof we need to show that if $X$ is bounded, then it is connected. If not, there would be two disjoint nonempty clopen subsets $Y$ and $X\setminus Y$ of $X$. But then there would be $\varepsilon>0$ with $Y(\varepsilon)\cap (X\setminus Y)=\emptyset$ which by Lemma \ref{LemBoundedSeparatedComponnent} leads to a contradiction. 
\end{proof}

\section{Theorem \ref{CharacterizationOfLimitSets} does not hold in infinitely dimensional spaces}

Now we will define: a conditionally convergent series $\sum _{n=1}^{\infty}y_{n}$ in $c_{0}$ such that $\sum _{n=1}^{\infty} y_{n}=\theta:=(0,0,\ldots)$ and its rearrangement $\sigma$ such that $\on{LIM}(\sum _{n=1}^{\infty} y_{\sigma(n)})$ consists of two points. 
 
\begin{example}\label{Ex3}
\emph{
As in Example \ref{Ex1} and Example \ref{Ex2} the constructed series  $\sum _{n=1}^{\infty} y_{n}$ will be alternating, so we will define only elements with odd indexes.  Let $\{e_{i}:i\in\N\}$ be a standard basis in $c_0$. We define $(y_n)$ inductively:\\ 
\emph{Step 1.} At first we define three odd elements; the first equals $e_2$, the second equals $e_1$, and the third equals $-e_2$;\\
\emph{Step k+1.}  In this step we define $3\cdot 2^{k}$ elements of the series with odd indexes: the first $2^{k}$ of them are equal to $\frac{1}{2^{k}}e_{k+2}$, next $2^{k}$ of $y_{n}'s$ equal $\frac{1}{2^{k}}e_{1}$ and the last $2^{k}$ of them are equal to $-\frac{1}{2^{k}}e_{k+2}$.\\
We define a rearrangement $\sum_{n=1}^\infty x_n$ of $\sum_{n=1}^\infty y_n$ in the similar way as in Example \ref{Ex1}. Namely, $x_{1},x_2,x_{3}$ are consecutive elements of $(y_{n})$ defined in Step 1 with odd indexes, that is $x_{1}=y_{1},x_{2}=y_{3},x_{3}=y_{5}$. Next three of $x_{n}'s$ are the elements of $(y_{n})$  with even indexes corresponding to the odd indexes defined in Step 1 with reversed order, that is $x_{4}=y_{6},x_{5}=y_{4},x_{6}=y_{2}$. In Step $k+1$ the first $3\cdot 2^{k}$ $x_n$'s are consecutive elements of $(y_{n})$ defined in Step $k+1$ with odd indexes, and next $3\cdot 2^{k}$ $x_n$'s are consecutive elements of $(y_{n})$ with even indexes with reversed order. The sequence of partial sums $s_{m}=\sum _{n=1}^{m}x_{n}$ is the following
$$
e_2,e_2+e_1,e_1,e_2+e_1,e_2,\theta,
$$
$$
\frac12 e_3,e_3,e_3+\frac12 e_1,e_3+e_1,\frac12 e_3+e_1,e_1,\frac12 e_3+e_1,e_3+e_1,e_3+\frac12 e_1,e_3,\frac12 e_3,\theta,\dots
$$
$$
\frac{1}{2^k}e_{k+2},\frac{2}{2^k}e_{k+2},\dots,e_{k+2},e_{k+2}+\frac{1}{2^k}e_1,\dots,e_{k+2}+e_1,\frac{2^k-1}{2^k}e_{k+2}+e_1,\dots,e_1,
$$
$$
e_1+\frac{1}{2^k}e_{k+2},\dots,e_1+e_{k+2}, e_{k+2}+\frac{2^k-1}{2^k}e_1,\dots,e_{k+2},\frac{2^k-1}{2^k}e_{k+2},\dots, \theta,\dots
$$
}

\emph{The walk $(s_{m})$ has the following properties:\\
(i) $\theta$ and $e_1$ appear infinitely many times in a partial sums sequence $(s_m)$;\\
(ii) for every natural number $j\geq 2$ there exists $p\in\mathbb{N}$ such that $s_{m}(j)=0$ for every natural $m\geq p$;\\
(iii) the distance between the point $z=(z(1),z(2),\ldots)$ with $z(1)\notin [0,1]$ and the set $\{s_{m}:m\in\mathbb{N}\}$ is positive;\\
(iv) if $s_m(1)\notin\{0,1\}$ then there exists a natural $k\geq 2$ such that $s_m(k)=1$.}

\emph{We claim that $\on{LIM}(\sum _{n=1}^{\infty} x_{n})=\{\theta,e_1\}$. By (i) we get $\theta,e_1\in\on{LIM}(\sum _{n=1}^{\infty} x_{n})$. Conditions (ii)  and (iii) give us the inclusion  $\on{LIM}(\sum _{n=1}^{\infty} x_{n})\subseteq \{(a,0,0,\ldots) : a\in[0,1]\}$. Indeed, since $z=(z(1),z(2),\ldots)\in\on{LIM}(\sum _{n=1}^{\infty} x_{n})$ then by (ii) we get $z(i)=0$ for every $i\geq 2$. Moreover, if $z(1)>1$ or $z(1)<0$ then by (iii) we have $z\notin\on{LIM}(\sum _{n=1}^{\infty} x_{n})$. Now, let $a\in(0,1)$. We will show that $(a,0,0,\ldots)\notin\on{LIM}(\sum _{n=1}^{\infty} x_{n})$. One can find $\varepsilon>0$ such that $(a-\varepsilon,a+\varepsilon)\cap\{0,1\}=\emptyset$. We consider the ball $B((a,0,0,\ldots),\varepsilon)$ in $c_{0}$. If $z\in B((a,0,0,\ldots),\varepsilon)\cap\{s_{m}:m\in\mathbb{N}\}$ then $z(1)\in(a-\varepsilon,a+\varepsilon)$, hence the first coordinate of $z$ is neither $0$ nor $1$. Then by (iv) there exists a natural number $k\geq 2$ such that $z(k)=1$ which contradicts the fact that $z\in B((a,0,0,\ldots),\varepsilon)$. Hence  $B((a,0,0,\ldots),\varepsilon)\cap\{s_{m}:m\in\mathbb{N}\}=\emptyset$, so $(a,0,0,\ldots)\notin\overline{\{s_{m}:m\in\mathbb{N}\}}$. Since $\on{LIM}(\sum _{n=1}^{\infty} x_{n})$ is contained in $\overline{\{s_{m}:m\in\mathbb{N}\}}$, we have $(a,0,0,\ldots)\notin\on{LIM}(\sum _{n=1}^{\infty} x_{n})$. Finally $\on{LIM}(\sum _{n=1}^{\infty} x_{n})=\{\theta,e_1\}$.
}
\end{example}

{\bf Remark.} Roman Witu\l a reminded us that he had found a very similar example of series with two-point limit set, see \cite{WHK}.  

\section{On the reverse of Theorem \ref{CharacterizationOfLimitSets}}

In this Section we will prove that Theorem \ref{CharacterizationOfLimitSets} can be reversed. It means that for any compact and connected subset $X$ of a Euclidean space $\R^m$ there is a conditionally convergent series $\sum_{n=1}^\infty x_n$ and a permutation $\sigma\in S_\infty$ with $X=\on{LIM}(\sum_{n=1}^\infty x_{\sigma(n)})$, and for any closed subset $Y$ of $\R^m$ whose each component is unbounded there is a conditionally convergent series $\sum_{n=1}^\infty y_n$ and permutation $\tau\in S_\infty$ with $Y=\on{LIM}(\sum_{n=1}^\infty y_{\tau(n)})$. This shows that Theorem \ref{CharacterizationOfLimitSets} gives a full characterization of limit sets in finitely dimensional Banach spaces. 

\begin{theorem}\label{EpsilonChainable}
Let $m\in\N$. Assume that $X\subseteq\R^m$ is closed and $\varepsilon$-chainable for every $\varepsilon>0$. Then there is a conditionally convergent series $\sum_{n=1}^\infty x_n$ in $\R^m$ such that $X=\on{LIM}(\sum_{n=1}^\infty x_{\sigma(n)})$ for some $\sigma\in S_\infty$. In particular, the assertion holds if $X$ is compact and connected.
\end{theorem}
\begin{proof}
Let $(d_{n})$ be dense in $X$. We will construct an $X$-walk. In the first step we find a 1-chain inside $X$ between points $d_{1}$ and $d_{2}$ and denote it $a_{1}=d_{1},a_{2},\ldots,a_{p}=d_{2}$. We define $s_{i}=a_{i}$ for every $i\in\{1,\ldots,p\}$. Then we go back to $d_{1}$ using the same way, which means that $s_{i}=a_{2p-i}$ for $i\in\{p+1,\ldots,2p-1\}$. In the second step let $a_{2p-1}=d_{1},a_{2p},\ldots,a_{2p-1+r}=d_{3}$ be a $2^{-1}$-chain between $d_{1}$ and $d_{3}$. We define $y_{n}$'s in the same way, that means they are the following elements of the next chain, $s_{i}=a_{i}$ for $i\in\{2p-1,\ldots,2p-1+r\}$ and then we go back to $d_{1}$ via the same elements. In the third step we consider a $2^{-2}$-chain between $d_{1}$ and $d_{4}$ and  define the next $s_{n}$'s as before, and so on. By Proposition \ref{X-walkProp}, we obtain the assertion. Finally, note that connected sets are $\varepsilon$-chainable for every $\varepsilon>0$. 
\end{proof}

\begin{theorem}\label{ReverseCharacterization}
Let $m\geq 2$. Assume that $X\subseteq\R^m$ is closed and any component of $X$ is unbounded. Then there is a conditionally convergent series $\sum_{n=1}^\infty x_n$ in $\R^m$ such that $X=\on{LIM}(\sum_{n=1}^\infty x_{\sigma(n)})$ for some $\sigma\in S_\infty$. 
\end{theorem}
\begin{proof}
Let $X=\bigcup _{t\in T}A_{t}$, where for every $t\in T$ the set $A_{t}$ is an unbounded component of $X$. Clearly each $A_{t}$ is closed and $\varepsilon$-chainable for every $\varepsilon>0$. Let $(d_{n})$ be dense in $X$.  If $d_i\in A_s$, $d_j\in A_t$, $A_t\cap A_s=\emptyset$, then, by the connectedness and the unboundedness of $A_s$ and $A_t$, there is a sphere $S(0,R)$ which has non-empty intersections with $A_s$ and $A_t$. Let $a_s\in A_s\cap S(0,R)$ and $a_t\in A_t\cap S(0,R)$. By an $\varepsilon$-chain via $S(0,R)$ from $d_i$ to $d_j$ we mean a concatenation of three $\varepsilon$-chains: from $d_i$ to $a_s$ using elements of $A_s$, from $a_s$ to $a_t$ using elements of $S_R$ and from $d_j$ to $a_t$ using elements of $A_t$. If $A_t=A_s$, then by an $\varepsilon$-chain via $S(0,R)$ from $d_i$ to $d_j$ we mean just an $\varepsilon$-chain from $d_i$ to $d_j$ using elements of $A_s$.  Let $R_n$ be a sequence of radii tending to $\infty$ such that $S(0,R_n)$ has a nonempty intersection with each component containing $d_1,\dots,d_{n+1}$. Now let us describe a walk $(s_n)$, which, in general, does not need to be an $X$-walk: The first elements of the walk $(s_n)$ is a $1/2$-chain via $S(0,R_1)$ from $d_1$ to $d_2$, and then back via the same elements. In the $k$-th step of the construction the next elements of $(s_n)$ are elements of a concatenation of $2^{-k}$-chains via $S(0,R_k)$ from $d_i$ to $d_{i+1}$, $i=1,\dots,k$, and then back from $d_{k+1}$ to $d_1$ via the same elements.

Using the same argument as in Proposition \ref{X-walkProp} we can find an alternating series $\sum_{n=1}^\infty x_n$ and $\sigma\in S_\infty$ such that $s_n=\sum_{i=1}^n x_{\sigma(i)}$. Clearly $X\subseteq\on{LIM}(\sum_{n=1}^\infty x_{\sigma(n)})\subseteq X\cup\bigcup_{k=1}^\infty S(0,R_k)$. Since $R_k\to\infty$ and the sequence $(s_n)$ contains at most finitely many elements of $S(0,R_k)\setminus X$, we obtain the reverse inclusion $X\supseteq\on{LIM}(\sum_{n=1}^\infty x_{\sigma(n)})$.
\end{proof}

As we have mentioned in Introduction, the assertion of Theorem \ref{ReverseCharacterization} is not true if $m=1$.

\section{When a limit set is a singleton}

By definition, if $\sum _{n=1}^{\infty}x_{n}=x_{0}$, then $\on{LIM}(\sum _{n=1}^{\infty}x_{n})=\{x_{0}\}$, since every subsequence of the sequence of partial sums is convergent to $x_{0}$. In general the inverse implication does not need to be true which is illustrated, in $c_0$, by Proposition \ref{rozbieznysingleton}. However, in finitely dimensional spaces the above implication can be reversed.   
\begin{theorem}
Let $\sum_{n=1}^\infty x_n$ be a series in $\R^m$ with $x_n\to 0$. If $\on{LIM}(\sum_{n=1}^\infty x_{n})$ is a singleton, then $\sum_{n=1}^\infty x_{n}$ is convergent and $\{\sum_{n=1}^\infty x_{n}\}=\on{LIM}(\sum_{n=1}^\infty x_{n})$.
\end{theorem}
\begin{proof}
Let $\on{LIM}(\sum _{n=1}^{\infty}x_{n})=\{x_{0}\}$. Suppose that $\sum _{n=1}^{\infty}x_{n}$ does not converge to $x_{0}$, so there exists $\varepsilon>0$ such that for every $k_{0}\in\mathbb{N}$ one can find $l\geq k_{0}$ such that $\Vert \sum _{n=1}^{l}x_{n}-x_{0}\Vert>\varepsilon$. That means that there are infinitely many indexes $p$ such that $\sum _{n=1}^{p}x_{n}\notin\overline{B}(x_{0},\varepsilon)$. On the other hand, since $x_0$ is a limit point of the series $\sum _{n=1}^{\infty}x_{n}$, there exist infinitely many $r\in\mathbb{N}$ such that $\sum _{n=1}^{r}x_{n}\in B(x_{0},\varepsilon/2)$. Hence there are infinitely many elements of a walk $(s_n)$ of partial sums inside the ball $B(x_{0},\varepsilon/2)$ and infinitely many outside the closed ball $\overline{B}(x_{0},\varepsilon)$. Since $x_n\to 0$, there are infinitely many $s_n$'s in $B=\overline{B}(x_{0},\varepsilon)\setminus B(x_{0},\varepsilon/2)$. By the compactness of $B$, it contains a limit point of $(s_n)$ which contradicts that $\on{LIM}(\sum _{n=1}^{\infty}x_{n})$ is a singleton. 
\end{proof}

Note that the assumption $x_n\to 0$ cannot be omitted. To see this consider the series $2^{-1}+2^1-2^1+2^{-2}+2^2-2^2+2^{-3}+2^3-2^3+\dots$. Clearly 1 is its only limit point, but the series is not convergent.

\begin{proposition}\label{rozbieznysingleton}
There is a conditionally convergent series $\sum_{n=1}^\infty x_n$ in $c_0$ and $\sigma\in S_\infty$ such that $\on{LIM}(\sum_{n=1}^\infty x_{\sigma(n)})=\{\theta\}$ but $\sum_{n=1}^\infty x_{\sigma(n)}$ diverges. 
\end{proposition}
\begin{proof}
To define an alternating series $\sum _{n=1}^{\infty}x_{n}$ it suffices to prescribe only the elements with odd indexes:\\ 
\emph{Step 1.} Firstly we define $x_{1}=e_{1}$.\\ 
\emph{Step k.} The next $2^{k-1}$ elements with odd indexes are equal to $\frac{1}{2^{k-1}}e_{k}$.\\
The series $\sum_{n=1}^\infty x_n$ is the following
$$
e_1-e_1+\frac{1}{2}e_2-\frac{1}{2}e_2+\frac{1}{2}e_2-\frac{1}{2}e_2+\frac{1}{4}e_3-\frac{1}{4}e_3+\frac{1}{4}e_3-\frac{1}{4}e_3
+\frac{1}{4}e_3-\frac{1}{4}e_3+\frac{1}{4}e_3-\frac{1}{4}e_3+\dots
$$
Now we define a rearrangement $\sum_{n=1}^\infty x_{\sigma(n)}$. Firstly we use elements with odd indexes defined in Step 1, and then the corresponding elements with even indexes; secondly we use elements with odd indexes defined in Step 2, and then the corresponding elements with even indexes, and so on. The rearranged series $\sum_{n=1}^\infty x_{\sigma(n)}$ has the following form
$$
e_1-e_1+\frac{1}{2}e_2+\frac{1}{2}e_2-\frac{1}{2}e_2-\frac{1}{2}e_2+\frac{1}{4}e_3+\frac{1}{4}e_3+\frac{1}{4}e_3+\frac{1}{4}e_3
-\frac{1}{4}e_3-\frac{1}{4}e_3-\frac{1}{4}e_3-\frac{1}{4}e_3+\dots
$$
Thus the walk $s_n=\sum_{k=1}^nx_{\sigma(k)}$ has the following properties: 
\begin{enumerate}
\item $s_{2^{k+1}-2}=\theta$ for every $k\in\mathbb{N}$;
\item for every natural number $j$ exists $p\in\mathbb{N}$ such that $s_{m}(j)=0$ for every $m\geq p$;
\item $s_{2^{k}+2^{k-1}-2}=e_{k}$ for every $k\in\mathbb{N}$.
\end{enumerate}
From (1) we have the inclusion $\{\theta\}\subseteq\on{LIM}(\sum _{n=1}^{\infty}x_{\sigma(n)})$ and from (2) we get the inverse inclusion. Hence $\on{LIM}(\sum _{n=1}^{\infty}x_{\sigma(n)})=\{\theta\}$. From (1) and (3) we get $\Vert s_{2^{k}+2^{k-1}-2}(k)-s_{2^{k+1}-2}(k)\Vert=\Vert e_{k}\Vert=1$ for every $k\in\mathbb{N}$. That means that the sequence of partial sums of the rearranged series is not a Cauchy sequence, and consequently it diverges.
\end{proof}

\section{Rearrangement property}

Klee proved that if $A\subseteq\on{SR}(\sum_{n=1}^\infty x_n)$ is closed and $\varepsilon$-chainable for every $\varepsilon>0$, then there is $\tau\in S_\infty$ such that $A=\on{LIM}(\sum_{n=1}^\infty x_{\tau(n)})$. This is a strengthening of Theorem \ref{EpsilonChainable} -- to see it take any conditionally convergent series $\sum_{n=1}^\infty x_n$ with $\on{SR}(\sum_{n=1}^\infty x_n)=\R^m$. We show that this fact holds true in every Banach space provided $\sum_{n=1}^\infty x_n$ has the so-called Rearrangement Property. In fact we prove that if $A\subseteq\overline{\on{SR}(\sum_{n=1}^\infty x_n)}$ is closed and $\varepsilon$-chainable for every $\varepsilon>0$, then there is $\tau\in S_\infty$ such that $A=\on{LIM}(\sum_{n=1}^\infty x_{\tau(n)})$. As a byproduct of this observation we obtain that if  $\sum_{n=1}^\infty x_n$ has the Rearrangement Property, then its sum range $\on{SR}(\sum_{n=1}^\infty x_n)$ is closed.

\begin{lemma}\label{przesuniecie}
Let $\sum _{n=1}^{\infty} x_{n}$ be a conditionally convergent series in a Banach space $X$. Then $SR(\sum _{n=1}^{\infty} x_{n})=SR(\sum _{n=k+1}^{\infty} x_{n})+\sum _{n=1}^{k} x_{n}$ for every $k\in\mathbb{N}$.
\end{lemma}  

\begin{proof}
"$\supseteq$" Let $k\in\mathbb{N}$ and $x\in SR(\sum _{n=k+1}^{\infty} x_{n})+\sum _{n=1}^{k} x_{n}$. Then there exists a permutation $\sigma:[k+1,\infty)\rightarrow[k+1,\infty)$ such that $x=\sum _{n=k+1}^{\infty} x_{\sigma(n)}+\sum _{n=1}^{k} x_{n}$. Define $\pi(n)=n$ for $n\leq k$ and $\pi(n)=\sigma(n)$ for $n\geq k+1$. Hence $x=\sum _{n=1}^{\infty} x_{\pi(n)}$, so $x\in SR(\sum _{n=1}^{\infty} x_{n})$.

"$\subseteq$" Let $x\in SR(\sum _{n=1}^{\infty} x_{n})$ and $k\in\mathbb{N}$. Then there exists a permutation $\pi:\mathbb{N}\rightarrow\mathbb{N}$ such that $x=\sum _{n=1}^{\infty} x_{\pi(n)}$. Let $M=\pi^{-1}(\{1,\ldots,k\})$. Then
for every $\varepsilon>0$ there exists $m_0\geq\max M$ such that for every $m>m_0$ the following inequality is true: $\Vert x-\sum _{n=1}^{m} x_{\pi(n)}\Vert<\varepsilon$. It means that  $\Vert x-\sum _{n=1}^{k} x_{n}-\sum _{n\in\{1,\ldots,m\}\setminus M} x_{\pi(n)}\Vert<\varepsilon$ for every $m>m_0$. Define a permutation $\sigma:[k+1,\infty)\rightarrow[k+1,\infty)$ as follows $\sigma(k+l)=\pi(n)$ where $n$ is the $l$-th number in the set $\N\setminus M$. Then $x=\sum_{n=k+1}^\infty x_{\sigma(n)}+\sum_{n=1}^kx_n$. Hence $x\in SR(\sum _{n=k+1}^{\infty} x_{n})+\sum _{n=1}^{k} x_{n}$.
\end{proof}

We say that a conditionally convergent series  $\sum _{k=1}^{\infty} x_{k}$ has the Rearrangement Property, or (RP), if for every $\varepsilon>0$ there are: a natural number $N(\varepsilon)$ and a positive real number $\delta(\varepsilon)$ such that the implication
$$
\Vert \sum _{i=1}^{n} y_{i}\Vert<\delta(\varepsilon)\Rightarrow\Big(\max _{j\leq n}\Vert \sum _{i=1}^{j} y_{\sigma(i)}\Vert<\varepsilon\text{ for some permutation }\sigma\in S_{n}\Big)
$$
holds for every finite sequence $(y_{i})_{i=1}^{n}\subseteq(x_{i})_{i=N(\varepsilon)}^{\infty}$. Note that if $\varepsilon>\varepsilon'>0$, then we can find numbers $\delta(\varepsilon),N(\varepsilon)$ and $\delta(\varepsilon'),N(\varepsilon')$ from the definition of (RP) used for $\varepsilon$ and $\varepsilon'$, respectively, such that $\delta(\varepsilon)\geq\delta(\varepsilon')$ and $N(\varepsilon)\leq N(\varepsilon')$. Similarly, having a decreasing sequence $(\varepsilon_n)$ of positive real numbers, we find $\delta(\varepsilon_n),N(\varepsilon_n)$ from the definition of (RP) such that $\delta(\varepsilon_n)\geq\delta(\varepsilon_{n+1})$ and $N(\varepsilon_n)\leq N(\varepsilon_{n+1})$ for every $n\in\N$. 

\begin{lemma}\label{przedluzenie}
Assume that $\sum _{n=1}^{\infty} x_{n}$ is a conditionally convergent series with (RP) in a Banach space $X$. Let $\varepsilon\geq\varepsilon'>0$ and let $\delta(\frac{\varepsilon}{2}),N(\frac{\varepsilon}{2})$ and $\delta(\frac{\varepsilon'}{2}),N(\frac{\varepsilon'}{2})$ be numbers from the definition of (RP) used for $\frac{\varepsilon}{2}$ and $\frac{\varepsilon'}{2}$, respectively. Let $k\in\mathbb{N}$, $a,b\in SR(\sum _{n=1}^{\infty} x_{n})$ with $\Vert a-b\Vert<\min\{\frac{\varepsilon}{12},\frac{1}{3}\cdot\delta(\frac{\varepsilon}{2})\}$, and $\tau:[1,k]\rightarrow\mathbb{N}$ be a partial permutation such that $\Vert\sum _{n=1}^{k} x_{\tau(n)}-a\Vert\leq\min\{\frac{\varepsilon}{12},\frac{1}{3}\cdot\delta(\frac{\varepsilon}{2})\}$ and $\on{rng}\tau\supseteq[1,N(\frac{\varepsilon}{2})]$.
Then there exist $k'>k$ and a partial permutation $\tau':[1,k']\rightarrow\mathbb{N}$ such that the following conditions hold:
\begin{enumerate}
\item $\tau'|_{[1,k]}=\tau$ and $[1,\max \on{rng}\tau]\subseteq \on{rng}{\tau'}$;
\item $\Vert \sum _{n=1}^{p} x_{\tau'(n)}-a\Vert\leq\varepsilon$ for $p\in[k+1,k']$;
\item $\Vert \sum _{n=1}^{k'} x_{\tau'(n)}-b\Vert\leq\min\{\frac{\varepsilon'}{12},\frac{1}{3}\cdot
\delta(\frac{\varepsilon'}{2})\}$;
\item $\on{rng}(\tau')\supseteq[1,N(\frac{\varepsilon'}{2})]$.
\end{enumerate}
\end{lemma}
\begin{proof}
Let $k_{0}=\max\{N(\frac{\varepsilon'}{2}),N(\frac{\varepsilon}{2}),\max\on{rng}\tau\}$. Define $y=\sum _{n=1}^{k} x_{\tau(n)}$ and $z=\sum _{n\in\{1,\ldots,k_{0}\}\setminus\{\tau(1),\ldots,\tau(k)\}}x_{n}$. Hence $y+z=\sum _{n=1}^{k_{0}}x_{n}$.
From the assumption that $b\in SR(\sum _{n=1}^{\infty}x_{n})$ by Lemma \ref{przesuniecie} we obtain $b-(y+z)\in SR(\sum _{n=k_{0}+1}^{\infty}x_{n})$. Thus we can find $k_0<n_{1}<n_{2}<\ldots<n_{l}$ such that $y+z+w\in B\big(b,\min\{\frac{\varepsilon'}{12},\frac{1}{3}\cdot \delta(\frac{\varepsilon'}{2}),\frac{1}{3}\cdot\delta(\frac{\varepsilon}{2})\}\big)$, where $w=x_{n_{1}}+\ldots+x_{n_{l}}$.

Enumerate the set $([1,k_{0}]\setminus\{\tau(1),\ldots,\tau(k)\})\cup\{n_{1},\ldots,n_{l}\}$ as $\{m_{1}<m_2<\ldots<m_{k'-k}\}$, where $k'=k_{0}+l$. Hence,
$$
\Big\Vert\sum _{i=1}^{k'-k} x_{m_{i}}\Big\Vert=\Vert z+w\Vert\leq\Vert y-a\Vert+\Vert a-b\Vert +\Vert b-(y+z+w)\Vert.
$$
Consequently, 
$$
\Big\Vert\sum _{i=1}^{k'-k} x_{m_{i}}\Big\Vert\leq \min\Big\{\frac{\varepsilon}{12},\frac{1}{3}\cdot\delta\Big(\frac{\varepsilon}{2}\Big)\Big\}+\min\Big\{\frac{\varepsilon}{12},\frac{1}{3}\cdot\delta\Big(\frac{\varepsilon}{2}\Big)\Big\}+\min\Big\{\frac{\varepsilon'}{12},\frac{1}{3}\cdot\delta\Big(\frac{\varepsilon'}{2}\Big),\frac{1}{3}\cdot\delta\Big(\frac{\varepsilon}{2}\Big)\Big\}\leq\delta\Big(\frac{\varepsilon}{2}\Big).
$$
Since $m_{i}\geq N(\frac{\varepsilon}{2})$ for $i\in[1,k'-k]$ and $\Vert\sum _{i=1}^{k'-k} x_{m_{i}}\Vert\leq\delta(\frac{\varepsilon}{2})$, by (RP) there is a permutation $\sigma\in S_{k'-k}$ such that $\Vert\sum _{i=1}^{j} x_{m_{\sigma(i)}}\Vert\leq\frac{\varepsilon}{2}$ for every $j\in[1,k'-k]$. Let us define $\tau'(n)=\tau(n)$ for $n\leq k$ and $\tau'(n)=m_{\sigma(n-k)}$ for $n\in[k+1,k']$. Then for every $p\in[k+1,k']$ we have the following:
$$
\Big\Vert\sum _{n=1}^{p} x_{\tau'(n)}-a\Big\Vert=\Big\Vert\sum _{n=1}^{k} x_{\tau(n)}+\sum _{n=k+1}^{p} x_{\tau'(n)}-a\Big\Vert\leq \Vert y-a\Vert+\Big\Vert\sum _{n=k+1}^{p} x_{\tau'(n)}\Big\Vert\leq \min\Big\{\frac{\varepsilon}{12},\frac{1}{3}\cdot\delta\Big(\frac{\varepsilon}{2}\Big)\Big\}+\frac{\varepsilon}{2}<\varepsilon
$$
which gives us (2).

Now we check  (1), (3) and (4). Note that the numbers $1,\ldots,k_{0}$ are among $\tau'(1),\ldots,\tau'(k')$ and $k_{0}\geq\max \on{rng}\tau$. Therefore we have (1). Since $\sum _{n=1}^{k'}x_{\tau'(n)}=y+z+w$ and $\Vert y+z+w-b\Vert\leq\min\{\frac{\varepsilon'}{12},\frac13\delta(\frac{\varepsilon'}{2})\}$,  we obtain (3). Condition (4) follows from the fact that if $n\notin\on{rng}(\tau')$, then $n>k_{0}\geq N(\frac{\varepsilon'}{2})$.
\end{proof}

\begin{lemma}\label{osrodek}
Let $A$ be a subset of a Banach space such that $A$ is separable and $\varepsilon$-chainable for every $\varepsilon>0$. Let $(\eta_{i})$ be a sequence of positive numbers. Then there is a sequence $(d_{n})$ dense in $A$ with the property that there is an increasing sequence $(l_{i})$ such that $\{d_{l_{i}},d_{l_{i}+1},\ldots,d_{l_{i+1}}\}$ is an $\eta_{i}$-chain for every $i$.
\end{lemma}

\begin{proof}
Since $A$ is separable, there are $v_{1},v_{2},\ldots$ such that $A=\overline{\{v_{n}: n\in\mathbb{N}\}}$. Then one can find an $\eta_{i}$-chain: $d_{l_{i}},d_{l_{i}+1},\ldots,d_{l_{i+1}}$ of elements of $A$ with $d_{l_{i}}=v_{i}$ and $d_{l_{i+1}}=v_{i+1}$ for any $i\in\mathbb{N}$. Clearly the sequence $\{d_{n}\}_{n=1}^{\infty}$ fulfills the desired condition.
\end{proof}

\begin{lemma}\label{ApproxOfDenseSet}
Let $A$ bea  separable and $\varepsilon$-chainable for every $\varepsilon>0$ subset of a Banach space. Assume that $\{d_i:i\in\N\}$ is a dense subset of $A$ and $(\varepsilon_i)$ is a sequence of positive numbers tending to zero. If $(x_i)$ is such that $\Vert x_i-d_i\Vert<\varepsilon_i$ for every $i\in\N$, then $\on{LIM}(x_i)=\overline{A}$ where $\on{LIM}(x_i)$ denotes the set of all limit points of the sequence $(x_i)$. 
\end{lemma}

\begin{proof}
If $A$ is a singleton, then $d_i=a$, $A=\{a\}$ and $x_i\to a$. Then $\on{LIM}(x_i)=\{a\}=\overline{A}$. Assume that $A$ has at least two elements. Clearly $A$ is dense-in-itself. Fix $i\in\N$. There is a sequence $(d_{j_k})_{k=1}^\infty$ such that $j_1<j_2<\dots$ and $\Vert d_{j_k}-d_i\Vert<\varepsilon_{j_k}$. Then $x_{j_k}\to d_i$ and consequently $d_i\in\on{LIM}(x_i)$. Since the set $\on{LIM}(x_i)$ is closed, we have $\overline{A}\subseteq\on{LIM}(x_i)$. 

Note that for every $k$ almost every element of $(x_i)$ is in $\varepsilon_k$-shell of $A$. Thus  $\overline{A}\supseteq\on{LIM}(x_i)$.
\end{proof}

\begin{theorem}\label{mainThm}
Let  $\sum _{n=1}^{\infty} x_{n}$ be a conditionally convergent series with (RP) in a Banach space $X$. Then for every $A\subseteq\overline{\on{SR}(\sum_{n=1}^{\infty} x_{n})}$ which is closed and $\varepsilon$-chainable for every $\varepsilon>0$, there exists a permutation $\tau\in S_\infty$ such that $A=\on{LIM}(\sum _{n=1}^{\infty}x_{\tau(n)})$. 
\end{theorem}

\begin{proof}
Let $\varepsilon_{i}=\frac{1}{2^{i}}$. We fix numbers $\delta(\frac{\varepsilon_i}{2}),N(\frac{\varepsilon_i}{2})$ from the definition of (RP) such that $\delta(\frac{\varepsilon_i}{2})\geq\delta(\frac{\varepsilon_{i+1}}{2})$ and $N(\frac{\varepsilon_i}{2})\leq N(\frac{\varepsilon_{i+1}}{2})$ for every $i\in\N$. Since $A$ is separable and $\varepsilon$-chainable for every $\varepsilon>0$, using Lemma \ref{osrodek}, let $A=\overline{\{d_{n} :n\in\mathbb{N}\}}$, where for every $i\in\mathbb{N}$ the elements $d_{l_{i}},d_{l_{i}+1},\ldots,d_{l_{i+1}}$ form a $\eta_{i}$-chain for some $1=l_{1}<l_{2}<\ldots$, where $\eta_{i}=\min\{\frac{\varepsilon_{i}}{48},\frac{1}{12}\cdot\delta(\frac{\varepsilon_{i}}{2})\}$. Note that $(\eta_i)$ is a non-increasing sequence of positive real numbers. 

Inductively we define natural numbers $1=k_1<k_2<\dots$, one-to-one functions $\tau_i:[1,k_{i+1}]\to\N$ and $d_1',d_2',\dots$ fulfilling the following conditions
\begin{itemize}
\item[(i)] $\tau_i\subseteq\tau_{i+1}$;
\item[(ii)] $[1,\max\on{rng}\tau_i]\subseteq\on{rng}(\tau_{i+1})$;
\item[(iii)] $\Vert\sum_{n=1}^p x_{\tau_i(n)}-d_{i-1}'\Vert<\varepsilon_j$ for $p\in[k_i+1,k_{i+1}]$ and $i\in[l_j+1,l_{j+1}]$;
\item[(iv)] $d_i'\in\on{SR}(\sum_{n=1}^\infty x_n)$, $\Vert d_i'-d_i\Vert<\eta_j$ where $i\in[l_j,l_{j+1}-1]$;
\item[(v)] $\Vert\sum_{n=1}^{k_{i+1}}x_{\tau_i(n)}-d_i'\Vert<4\eta_j$ where $i\in[l_j,l_{j+1}-1]$;
\item[(vi)] $\on{rng}\tau_{i}\supseteq[1,N(\frac{\varepsilon_j}{2})]$ where $i\in[l_j,l_{j+1}-1]$. 
\end{itemize}

Define $x=\sum _{n=1}^{N(\varepsilon_{1}/2)}x_{n}$. Let $d_{1}'\in SR(\sum _{n=1}^{\infty} x_{n})$ be such that $\Vert d_{1}-d_{1}'\Vert<\eta_{1}$. Hence from Lemma \ref{przesuniecie} we get $d_{1}'-x\in SR(\sum _{N(\frac{\varepsilon_{1}}{2})+1}^{\infty}x_{n})$. Let $\pi:[N(\frac{\varepsilon_{1}}{2})+1,\infty)\rightarrow[N(\frac{\varepsilon_{1}}{2})+1,\infty)$ be a bijection such that $d_{1}'-x=\sum _{n=N(\varepsilon_{1}/2)+1}^{\infty} x_{\pi(n)}$. One can find a natural number $k_{2}>N(\frac{\varepsilon_{1}}{2})$, which satisfies the inequality $\Vert d_{1}'-x-\sum _{n=N(\varepsilon_{1}/2)+1}^{k_{2}} x_{\pi(n)}\Vert\leq\eta_{1}$. Define $\tau_1(k)=k$ for $k\leq N(\frac{\varepsilon_{1}}{2})$ and $\tau_1(k)=\pi(k)$ for $k\in[N(\frac{\varepsilon_{1}}{2})+1,k_{2}]$. Conditions (i)--(vi) are fulfilled for $\tau_1,d_1',k_1,k_2$. We do not need to check (i) and (ii), condition (iii) needs to be checked for $i\geq l_1+1=2$ and conditions (iv)--(vi) are fulfilled since $l_1=1$. 

Assume now, that we have already defined $\tau_1,\dots,\tau_i$, $k_1<\dots<k_{i+1}$ and $d_1',\dots,d_i'$ fulfilling (i)--(vi). Find $d_{i+1}'$ such that (iv) holds. We use Lemma \ref{przedluzenie} for $a=d_i'$, $b=d_{i+1}'$, $\tau=\tau_i$, $\varepsilon=\varepsilon_j$ where $l_j\leq i\leq l_{j+1}-1$, $\varepsilon'=\varepsilon_q$ where $l_q-1\leq i\leq l_{q+1}-2$, and $k=k_{i+1}$; note that $j=q$ if $l_j\leq i<l_{j+1}-1$, that is if $i\notin\{l_{s}-1:s\geq 1\}$, otherwise $i=l_{j+1}-1$ implies that $q=j+1$. Let us check the assumptions of Lemma \ref{przedluzenie}. By (iv) and (vi) we obtain $a,b\in\on{SR}(\sum_{n=1}^\infty x_n)$ and $\on{rng}\tau\supseteq[1,N(\frac{\varepsilon}{2})]$. Since $d_{l_j},d_{l_j+1},\dots,d_{l_{j+1}}$ form a $\eta_j$-chain, by (iv) we obtain
$$
\Vert a-b\Vert\leq\Vert d_i-d_i'\Vert+\Vert d_i-d_{i+1}\Vert+\Vert d_{i+1}-d_{i+1}'\Vert\leq\eta_j+\eta_j+\eta_j<4\eta_j=\min\Big\{\frac{\varepsilon}{12},\frac13\cdot\delta\Big(\frac{\varepsilon}{2}\Big)\Big\}.
$$
By (v) we obtain $\Vert\sum_{n=1}^kx_{\tau(n)}-a\Vert\leq 4\eta_j=\min\{\frac{\varepsilon}{12},\frac13\cdot\delta(\frac{\varepsilon}{2})\}$.  Now, using Lemma \ref{przedluzenie} we find $k_{i+2}>k_{i+1}$ and function $\tau_{i+1}:[1,k_{i+2}]\to\N$ such that 
\begin{enumerate}
\item $\tau_{i+1}|_{[1,k_{i+1}]}=\tau_i$ and $[1,\max\on{rng}\tau_i]\subseteq \on{rng}\tau_{i+1}$;
\item $\Vert \sum_{n=1}^{p} x_{\tau_{i+1}(n)}-d_i'\Vert\leq\varepsilon_j$ where $p\in[k_{i+1}+1,k_{i+2}]$ and $i\in[l_j,l_{j+1}-1]$;
\item $\Vert \sum _{n=1}^{k_{i+2}} x_{\tau_{i+1}(n)}-d_{i+1}'\Vert\leq 4\eta_q$ where $i\in[l_q-1,l_{q+1}-2]$;
\item $\on{rng}(\tau_{i+1})\supseteq[1,N(\frac{\varepsilon_{q}}{2})]$ where $i\in[l_q-1,l_{q+1}-2]$.
\end{enumerate}
Note that $\tau_1,\dots,\tau_{i+1}$, $k_1<\dots<k_{i+2}$ and $d_1',\dots,d_{i+1}'$ fulfill (i)--(vi): By (1) we obtain conditions (i) and (ii). Since the condition $i+1\in[l_j+1,l_{j+1}]$ is equivalent to $i\in[l_j,l_{j+1}-1]$, we obtain (iii). The element $d_{i+1}'$ has already been chosen to fulfill (iv). Conditions (3) and (4) are exactly (v) and (vi) for $i+1$.  

Let $\tau=\bigcup_{i\geq 1}\tau_i:\N\to\N$. Then (i) implies that $\tau$ is one-to-one. Condition (ii) implies that $\tau$ is onto $\N$, and consequently $\tau\in S_\infty$. By (iii) and (iv) we obtain that the distance between $A$ and $\sum_{n=1}^p x_{\tau(n)}$ is less than $1/2^j$ for almost every $p\in\N$. Thus $\on{LIM}(\sum _{n=1}^{\infty}x_{\tau(n)})\subseteq\overline{A}$. By (iv) and (v) we obtain that $\Vert\sum_{n=1}^{k_{i+1}}x_{\tau(n)}-d_i\Vert<5\eta_j<\varepsilon_j$ where $i\in[l_j,l_{j+1}-1]$. Thus by Lemma \ref{ApproxOfDenseSet} we get $\overline{A}=\on{LIM}((\sum_{n=1}^{k_{i+1}}x_{\tau(n)})_{i=1}^\infty)\subseteq\on{LIM}(\sum _{n=1}^{\infty}x_{\tau(n)})$.
\end{proof}

It is well-known that every conditionally convergent series of elements in a finite dimensional Banach space has the (RP), for detail see \cite{Kadets}. Thus, the Klee's result which we have mentioned at the beginning of this Section is a particular case of Theorem \ref{mainThm}. Combining methods used in proofs of Theorem \ref{mainThm} and Theorem \ref{ReverseCharacterization} one can prove the following strengthening of Theorem \ref{ReverseCharacterization}.
\begin{corollary}
Let $m\geq 2$. Assume that $\on{SR}(\sum_{n=1}^\infty x_n)=\R^m$, $X\subseteq\R^m$ is closed and any component of $X$ is unbounded. Then $X=\on{LIM}(\sum_{n=1}^\infty x_{\sigma(n)})$ for some $\sigma\in S_\infty$.
\end{corollary}

Note that singletons are trivially $\varepsilon$-chainable for every $\varepsilon>0$. Fix $a\in\overline{\on{SR}(\sum_{n=1}^{\infty} x_{n})}$. Using Theorem \ref{mainThm} for $A\subseteq\overline{\on{SR}(\sum_{n=1}^{\infty} x_{n})}$ such that $A=\{a\}$, we obtain that there is $\tau\in S_\infty$ such that $\on{LIM}(\sum_{n=1}^{\infty} x_{\tau(n)})=\{a\}$. As we have seen in Proposition \ref{rozbieznysingleton} is does not necessarily mean that $\sum_{n=1}^{\infty} x_{\tau(n)}=a$. However, if we put $d_i=a$, then $d_i'\to a$ and by condition (iii) we get that almost all elements of the sequence $(\sum_{n=1}^{p} x_{\tau(n)})_{p=1}^\infty$ of partial sums are in every neighborhood of $a$. Therefore $\sum_{n=1}^{\infty} x_{\tau(n)}$ is convergent to $a$. Thus $a\in\on{SR}(\sum_{n=1}^{\infty} x_{n})$. Hence as a byproduct of the proof of Theorem \ref{mainThm} we obtain the following.
\begin{corollary}\label{CorMain}
Let  $\sum _{n=1}^{\infty} x_{n}$ be a conditionally convergent series in a Banach space $X$, which has the (RP). Then its sum range $\on{SR}(\sum_{n=1}^{\infty} x_{n})$ is a closed set. 
\end{corollary}

Now, we will discuss a problem whether or not Corollary \ref{CorMain} can be reversed, namely whether or not the closedness of the sum range $\on{SR}(\sum_{n=1}^{\infty} x_{n})$ implies the (RP) for a series $\sum_{n=1}^{\infty} x_{n}$. By $S_n$ we denote the set of all permutations of the set $[1,n]$. 
\begin{lemma}\label{LemmaNotRP}
Let $k\in\mathbb{N}$ and $n={2k \choose k}$, then there exists a finite sequence $x_{1},\ldots,x_{2k}\in\R^{n}$ such that:
\begin{enumerate}
\item $\Vert x_{i}\Vert_{\sup}=1$ for every $i\leq 2k$.
\item $\Vert\sum _{i=1}^{k} x_{\sigma(i)}\Vert_{\sup}\geq k$ for every $\sigma\in S_{2k}$.
\item $\sum _{i=1}^{2k} x_{i}=0$.
\end{enumerate}
\end{lemma}

\begin{proof}
Let $k\in\mathbb{N}$. There are $n=\left(2k\atop k\right)$ sequences of length $2k$ consisting of $k$ many $1$'s and $k$ many $-1$'s. Enumerate all such sequences as $t_1,\dots,t_n$. Define $x_i(j)=t_j(i)$ for $j=1,\dots,n$ and $i=1,\dots,2k$. Now, if $\sigma\in S_{2k}$, then there is a sequence $t_{j_{\sigma}}$ such that $t_{j_{\sigma}}(\sigma(i))=1$ for $i=1,\dots,k$ and $t_{j_{\sigma}}(\sigma(i))=-1$ for $i=k+1,\dots,2k$. Thus
$$
\sum_{i=1}^kx_{\sigma(i)}(j_{\sigma})=k,
$$
and consequently
$$
\Vert\sum_{i=1}^kx_{\sigma(i)}\Vert_{\sup}\geq k.
$$
\end{proof}

Before we state the last result, first note that if the series $\sum _{i=1}^{\infty} x_{i}$ does not have (RP), then one can find $\varepsilon>0$ such that for every $\delta>0$ and $N\in\mathbb{N}$ there exists the finite subsequence $\{y_{i}\}_{i=1}^{n}\subseteq \{x_{i}\}_{i=N}^{\infty}$ for which two conditions hold:
\begin{itemize}
\item $\Vert\sum _{i=1}^{n} y_{i}\Vert<\delta$.
\item for every $\sigma\in S_{n}$ there is $j\leq n$ such that $\Vert \sum _{i=1}^{j} y_{\sigma(i)}\Vert_{\sup}\geq\varepsilon$.
\end{itemize}
The following theorem shows that Corollary \ref{CorMain} cannot be reversed. 

\begin{theorem}
There is a conditionally convergent series $\sum_{n=1}^\infty z_n$ in $c_0$ such that it does not have (RP) and its sum range $\on{SR}(\sum_{n=1}^\infty z_n)$ is a singleton, in particular it is a closed set.
\end{theorem}

\begin{proof}
Define $e_{n}=(\delta_{in})_{i=1}^{\infty}$ for every $n\in\mathbb{N}$, where $\delta_{in}=1$, if $i=n$ and $\delta_{in}=0$ otherwise. Let $n_{0}=0$ and $n_{k}={2^{k+1} \choose 2^{k}}+n_{k-1}$. For every $k\in\mathbb{N}$ let $x_{1}^{(k)},\dots,x_{2^{k+1}}^{(k)}\in\R^{n_{k}-n_{k-1}}$ will be the sequence constructed in Lemma \ref{LemmaNotRP}. Define $y_{i}^{(k)}=\frac{1}{2^{k}}\cdot\sum _{j=1}^{n_{k}-n_{k-1}}x_{i}^{(k)}(j)\cdot e_{n_{k-1}+j}$ for $i,k\in\N$. It easy to see that $y_{i}^{(k)}\in c_{0}$. Define the series $\sum _{n=1}^{\infty}z_{n}$ as follows: 
$$
z_{1}=y_{1}^{(1)}, z_{2}=-y_{1}^{(1)}, z_{3}=y_{2}^{(1)}, z_{4}=-y_{2}^{(1)}, z_{5}=y_{3}^{(1)}, z_{6}=-y_{3}^{(1)}, z_{7}=y_{4}^{(1)}, z_{8}=-y_{4}^{(1)}, z_{9}=y_{1}^{(2)}, z_{10}=-y_{1}^{(2)},\ldots
$$ 
It easy to see that $\sum _{n=1}^{\infty}z_{n}$ converges to $0$. 

Let $\varepsilon=1$ and $N\in\mathbb{N}$, $\delta>0$.
One can find $k\in\mathbb{N}$, such that $n_{k-1}>N$. Then by Lemma \ref{LemmaNotRP} for $(y_{1}^{(k)},\ldots,y_{2^{k+1}}^{(k)})\subseteq(z_i)_{i\geq N}$ and every permutation $\sigma\in S_{2^{k+1}}$  we have
$$
\Vert\frac{1}{2^{k}}\cdot\sum _{i=1}^{2^{k}} y_{\sigma(i)}^{(k)}\Vert_{\sup}=
\Vert\frac{1}{2^{k}}\cdot\sum _{i=1}^{2^{k}} x_{\sigma(i)}^{(k)}\Vert_{\sup}\geq\frac{1}{2^{k}}\cdot 2^{k}=1=\varepsilon.
$$
Moreover $\Vert\sum_{i=1}^{2^{k+1}}y_i^{(k)}\Vert=0<\delta$. This proves that the series $\sum _{n=1}^{\infty}z_{n}$ does not have (RP). 

Since the projection of the series on each coordinate contains only finitely many nonzero terms and a finite sum does not change under rearrangements, then $\on{SR}(\sum_{n=1}^\infty z_n)=\{\theta\}$. 
\end{proof}

\textbf{Acknowledgment.} We would like to thank Filip Strobin for a very careful analysis of this article, shortening the proof of Lemma \ref{LemmaComponentsInBall} and bringing us Lemma \ref{ClosureIn1PointCompactification}. We would also like to thank Robert Stegli\'nski for fruitful discussions on our paper.

\end{document}